\documentclass[11pt, reqno]{amsart}
\usepackage{amsmath, amsthm, amscd, amsfonts, amssymb, graphicx, color}
\usepackage[bookmarksnumbered, colorlinks, plainpages]{hyperref}
\hypersetup{colorlinks=true,linkcolor=red, anchorcolor=green, citecolor=cyan, urlcolor=red, filecolor=magenta, pdftoolbar=true}
\newtheorem{theorem}{Theorem}[section]
\newtheorem{lemma}[theorem]{Lemma}

\newtheorem{corollary}[theorem]{Corollary}
\theoremstyle{definition}
\newtheorem{definition}[theorem]{Definition}
\newtheorem{example}[theorem]{Example}

\theoremstyle{remark}
\newtheorem{remark}[theorem]{Remark}
\numberwithin{equation}{section}

\begin{document}
\allowdisplaybreaks
\setcounter{page}{1}

\title[Spaceability on Some Classes of Banach Spaces]{Spaceability on Some Classes of Banach Spaces}

\author[A.R. Bagheri Salec, S. Ivkovi\'{c} and S. M. Tabatabaie]{AliReza Bagheri Salec, Stefan Ivkovi\'{c} and Seyyed Mohammad Tabatabaie}

\address{Department of Mathematics, University of Qom, Qom, Iran}
\email{\textcolor[rgb]{0.00,0.00,0.84}{r-bagheri@qom.ac.ir}}

\address{The Mathematical Institute of the Serbian Academy of Sciences and Arts,
	p.p. 367, Kneza Mihaila 36, 11000 Beograd, Serbia}
\email{\textcolor[rgb]{0.00,0.00,0.84}{stefan.iv10@outlook.com}}

\address{Department of Mathematics, University of Qom, Qom, Iran}
\email{\textcolor[rgb]{0.00,0.00,0.84}{sm.tabatabaie@qom.ac.ir}}

\address{
\newline
}

\subjclass[2010]{Primary 46E30; Secondary 43A15.\\
\indent $^{*}$Corresponding author}

\begin{abstract}
In this paper, we study spaceability of subsets of generalized Orlicz and Lebesgue spaces associated to a Banach function space. Also, we give some sufficient conditions for spaceability of subsets of a general Banach space which improves an important result on this topic. As an application, it is shown that the set of all bounded linear operators which are not positive semidefinite on a separable Hilbert space is spaceable.
\end{abstract} \maketitle 
\section{Introduction}
A subset of a topological vector space is called spaceable 
if its union with the singleton $\{0\}$ contains a closed infinite-dimensional linear subspace.  This concept was introduced in \cite{fon,aro} and so far has been considered by many researchers. 
As a useful tool, L. Bernal-Gonz\'alez and M.O. Cabrera in \cite[Theorem 2.2]{ber}  give some sufficient conditions for spaceability of the complement of a cone in a Banach function space. This  result covers some important ones proved in \cite{bot1, bot2}. By this tool, in \cite{SaTa19,SaTa20}
it is shown that the set $\mathcal M_q^p(\mathbb R^n)\backslash\bigcup_{q<r\leq p}\mathcal M_r^p(\mathbb R^n)$ is spaceable in the Morrey space $\mathcal M_q^p(\mathbb R^n)$, 
if $0<q<p<\infty$. Also, technically it is also proved  that  ${\rm w}{\mathcal M}^p_q({\mathbb R}^n) 
\setminus {\mathcal M}^p_q({\mathbb R}^n)$
is spaceable in the weak Morrey space ${\rm w}{\mathcal M}^p_q({\mathbb R}^n)$.
In \cite[Theorem 3.3]{kit}  D. Kitson and R. M. Timoney present another nice sufficient condition for a set to be spaceable in a Fr\'echet space. This topic has been studied in the context of some special sequence and function spaces in several papers 
(see \cite{ber2,ber1,bot3,bot1,bot2,gur,tab1} for example). 

In this paper, we focus on generalized Orlicz and Lebesgue spaces $X^\Phi$ and $X^p$ associated to a Banach function space $X$, where $\Phi$ is a Young function and $p\geq 1$. These structures were studied in \cite{cam,per4,per1,tab5} and contains usual Orlicz and Lebesgue spaces. Inspiring \cite{bot1,bot2} and as an extension of \cite[Theorem 3.3]{ber} we prove that if $X$ is a solid Banach function space, $\inf\{\|\chi_E\|_X: E\in \mathcal{A}_0\}=0$, and	$\sup\{\|\chi_E\|_X:\, E \in {\mathcal A}_\infty\}<\infty$, where ${\mathcal A}_\infty:=\{E \in {\mathcal A}:\, \chi_E\in X\}$, then for each $p\geq 1$, $X^p-\bigcup_{p<q}X^q$ is spaceable in $X^p$. This result is concluded from the technical Lemma \ref{Nielsen} which is a  generalization of \cite[Theorem 14.22]{niel}. In sequel, we give some necessary condition for inclusion of two generalized Orlicz space (as a generalization of \cite[Theorem 3 page 155]{rao}), and then prove that if the Young function $\Phi_2$ is not stronger than the other one $\Phi_1$, then $X^{\Phi_2}-X^{\Phi_1}$ is spaceable in $X^{\Phi_2}$. Finally, we give an abstract improvement of \cite[Theorem 2.2]{ber}. To emphasis the capacity of the obtained result, we apply it to  show that if $X$ is a solid Banach function space on $\Omega$ and $\inf\{\|\chi_E\|_X: E\in \mathcal{A}_0\}=0$, then for each $1\leq p,q< r$, the set $\{(f,g)\in X^p\times X^q:\,fg\notin X^r\}$ is spaceable in $X^p\times X^q$. As another application, we prove that the set of all bounded linear operators which are not positive semidefinite on a separable Hilbert space is spaceable. Moreover, it is shown that if $K$ is a two sided ideal cone in $B(\mathcal{H})$ and there exists a sequence of mutually disjoint subsets $\{J_n\}_{n\in\mathbb{N}}$ of $\mathbb{N}$ satisfying the condition $P_{J_n}KP_{J_n}\neq P_{J_n}B(\mathcal{H})P_{J_n}$ for all $n\in\mathbb{N}$, then $B(\mathcal{H})-K$ is spaceable in $B(\mathcal{H})$, where $\mathcal{H}$ is a separable Hilbert space with an orthonormal basis $\{e_j\}_{j\in\mathbb{N}}$, and $P_{J_n}$ is the orthogonal projection on the closed linear span of $\{e_j\}_{j\in J_n}$.
\section{Preliminaries}
In sequel, $(\Omega,{\mathcal A},\mu)$ is always a $\sigma$-finite measure space, and $\mathcal{M}_0(\Omega)$ is set of all $\mathcal{A}$-measurable complex-valued functions on $\Omega$.

	A linear subspace $X$ of $\mathcal M_0(\Omega)$ equipped with a given norm $\|\cdot\|_{ X}$ is called a \emph{Banach function space on $\Omega$} if $(X,\|\cdot\|_X)$ is a Banach space.  It is called \emph{solid} if for each $f\in  X$ and $g\in\mathcal M_0(\Omega)$  we have $g \in  X$
	and $\|g\|_{ X}\leq \|f\|_{ X}$ whenever $|g|\leq |f|$ a.e.

A convex function
$\Phi:[0,\infty)\rightarrow[0,\infty)$ is called a \emph{Young function} if 
$\Phi(0)=\lim\limits_{x\rightarrow 0}\Phi(x)=0$ 
and 
$\lim\limits_{x\rightarrow\infty}\Phi(x)=\infty$. 
	Let $ X$ be a Banach function space on $\Omega$. For each $f\in \mathcal{M}_0(\Omega)$ we put
	\begin{equation}
	\|f\|_\Phi:=\inf\left\{\lambda>0:\,\Phi(\frac{|f|}{\lambda})\in  X,\,\,\left\|\Phi(\frac{|f|}{\lambda})\right\|_ X\leq 1\right\}.
	\end{equation}
	Then, the set of all $f\in \mathcal{M}_0(\Omega)$ with $\|f\|_\Phi<\infty$ is denoted by $X^\Phi$.

As in \cite[Theorem 4.11]{cam}, $(X^\Phi,\|\cdot\|_\Phi)$ is a Banach function space on $\Omega$ (two functions in $X^\Phi$ which are equal almost everywhere is considered same). For each  $p\geq 1$, the function $\Phi_{0}$ defined by $\Phi_{0}(x):=x^p$ for all $x\geq 0$, is a Young function. Then, we denote $X^p:=X^{\Phi_{0}}$ and $\|\cdot\|_p:=\|\cdot\|_{\Phi_0}$. In particular, if $X:=L^1(\mu)$, then $X^\Phi=L^\Phi(\mu)$ and $X^p=L^p(\mu)$, the classical Orlicz and Lebesgue spaces.
\section{Main Results}
A subset $S$ of a Banach space $Y$ is called \emph{spaceable} in $Y$ if $S\cup\{0\}$ contains a closed  infinite-dimensional subspace of $Y$. In this section, first we study the spaceability of special subsets of $X^p$.   As in \cite{tab5}, for each function $f$ in $M_0(\Omega)$ we denote $E_f:=\{x\in \Omega: f(x)\neq0\}$. 
\begin{remark}\label{pcs}
Recall from \cite{ber} that a Banach function space $(\mathcal{E},\|\cdot\|)$ on $\Omega$ is a {\it PCS-space} if for each sequence $(f_n)$ with $f_n\rightarrow f$ in $\mathcal{E}$, there is a subsequence $(f_{n_k})$ of $(f_n)$ such that $f_{n_k}\rightarrow f$ a.e. This property plays a key role in the subject \emph{spaceability}. For instance, see Theorem \ref{32} below as a main result on this topic.  A Banach function space $X$ on a $\sigma$-finite measure space  $(\Omega,\mathcal{A},\mu)$ is PCS-space if and only if the embedding of $X$ into $\mathcal{M}_0(\Omega)$ is continuous, where $\mathcal{M}_0(\Omega)$ is equipped with the topology of convergence in measure on finite measure subsets.
If $X$ is a solid quasi-Banach function space on a $\sigma$-finite measure space 
then the embedding $X$ in $\mathcal{M}_0(\Omega)$ is always continuous, see \cite[Proposition 2.2 (i)]{okad} for the finite measure case.
\end{remark}
Next, we recall a result which was proved in \cite[Theorem 2.2]{ber}.
\begin{theorem}\label{32}
	 Let $(\mathcal{E},\|\cdot\|)$ be a Banach function space on $\Omega$ and $B$ be a nonempty
	subset of $\mathcal{E}$ such that:
	\begin{enumerate}	
	\item $\mathcal{E}$ is a PCS-space;
	\item there is a constant $k>0$ such that $\|f+g\|\geq k\,\|f\|$ for all $f,g\in \mathcal{E}$ with $E_f\cap E_g=\varnothing$;
	\item  $B$ is a cone;
	\item if $f,g\in \mathcal{E}$ such that $f+g\in B$ and  $E_f\cap E_g=\varnothing$ then $f,g\in B$;
	\item there is a sequence $\{f_n\}_{n=1}^\infty\subseteq \mathcal{E}-B$ such that for each distinct $m,n\in\mathbb{N}$, $E_{f_n}\cap E_{f_m}=\varnothing$.
	\end{enumerate}
Then, $\mathcal{E}-B$ is spaceable in $\mathcal{E}$.
\end{theorem}
For sequel, we need the next result which is a  generalization of \cite[Theorem 14.22]{niel}. The main idea for the proof comes from \cite[Theorem 14.22]{niel} but details are different.  The item ($a$) in this theorem is a more general version of the relation ($\alpha$) in \cite{ber}.
Denote 
$${\mathcal A}_0:=\{E \in {\mathcal A}:\, 0<\mu(E) \,\,\text{and}\,\,\chi_E\in X\}.$$


\begin{lemma}\label{Nielsen}
	Let $X$ be a solid Banach function space. Then, the followings are equivalent: 
		\begin{itemize}
	\item [($a$)] $\inf\{\|\chi_E\|_X: E\in \mathcal{A}_0\}=0$.
	\item [($b$)] There exists a sequence $\{A_n\}_{n=1}^\infty$ in $\mathcal{A}_0$ such that $A_n\cap A_m=\varnothing$ for all distinct $m,n\in\mathbb{N}$ and 
	 $$0<\|\chi_{A_n}\|_X\leq\frac{1}{2^n}, \quad (n\in \mathbb{N}).$$
	 \end{itemize}
\end{lemma}
\begin{proof}
	$(b)\Rightarrow (a)$: 
	Let $\{A_n\}_{n=1}^\infty$ be a sequence in $\mathcal{A}$ which satisfies in ($b$). Let $1\leq p<q<\infty$. We claim that $X^p\nsubseteq X^q$. By the assumptions, we can write
	\begin{equation}\label{55}
		\sum_{n=1}^{\infty}\left\|n\,\|\chi_{A_n}\|_X^{\frac{-p}{q}}\cdot\chi_{A_n}\right\|_X=\sum_{n=1}^{\infty}n\|\chi_{A_n}\|_X^{1-\frac{p}{q}}\leq \sum_{n=1}^{\infty}n(2^{\frac{p}{q}-1})^n<\infty.		
	\end{equation}
Set 
$$f:=\sum_{n=1}^{\infty}n^{\frac{1}{p}}\|\chi_{A_n}\|_X^{\frac{-1}{q}}\cdot\chi_{A_n}\quad\text{and}\qquad S_N:=\sum_{n=1}^{N} n^{\frac{1}{p}}\|\chi_{A_n}\|_X^{\frac{-1}{q}}\cdot\chi_{A_n}$$  for all $N\in \mathbb{N}$. By the relation \eqref{55} the  sequence $\{S_N^p\}_{N=1}^{\infty}$ is Cauchy in $X$ and so it converges to some $g\in X$ with norm topology, because $X$ is complete. Now, by Remark \ref{pcs}, there exists a subsequence of  $\{S_N^p\}_{N=1}^{\infty}$ that converges to $g$ a.e. Therefore,
$$g=\sum_{n=1}^{\infty}n\|\chi_{A_n}\|_X^{\frac{-p}{q}}\cdot\chi_{A_n}=\left(\sum_{n=1}^{\infty} n^{\frac{1}{p}}\|\chi_{A_n}\|_X^{\frac{-1}{q}}\cdot\chi_{A_n}\right)^p=f^p \quad {\rm a.e.}$$ 
This implies that $|f|^p\in X$, and so $f\in X^p$. On the other hand, in contrast, let $f\in X^q$. Then, since $X$ is solid we have
\begin{align*}
	\|f\|_q=&\||f|^q\|_X\\	
	=&\left\|\left(\sum_{n=1}^{\infty}n^{\frac{1}{p}}\,\|\chi_{A_n}\|_X^{\frac{-1}{q}}\cdot\chi_{A_n}\right)^q\right\|_X\\
	\geq &\left\|k^{\frac{q}{p}}\left\|\chi_{A_k}\right\|_X^{-1}\cdot\chi_{A_k}\right\|_X=k^{\frac{q}{p}}
\end{align*}
for all $k\in\mathbb{N}$, and this implies that $\|f\|_q=\infty$, a contradiction. Hence, $f\in X^p-X^q$. Now, thanks to  \cite[Theorem 2.1]{tab5} we have $\inf\{\|\chi_E\|_X: E\in \mathcal{A}_0\}=0$.

 $(a)\Rightarrow (b)$: Let $\inf\{\|\chi_E\|_X: E\in \mathcal{A}_0\}=0$.  For each $A\in \mathcal{A}$ put
$$\mathcal{K}(A):=\inf\{\|\chi_B\|_X: \,B\in \mathcal{A}_0,\, B\subseteq A\}.$$
Clearly, 
\begin{enumerate}
\item if $A_1, A_2\in \mathcal{A}$ and $A_1\subseteq A_2$, then $\mathcal{K}(A_2)\leq \mathcal{K}(A_1)$, and
\item for each $C,B\in\mathcal{A}$ with $B\subseteq C$, if $\mathcal{K}(B),\mathcal{K}(C-B)>0$, then $\mathcal{K}(C)>0$.
\end{enumerate}
Note that (2) holds since for each $E\in\mathcal{A}_0$, if $E\subseteq C$, then $\|\chi_E\|_X\geq \min\{\mathcal{K}(B),\mathcal{K}(C-B)\}$.

For each $A\in\mathcal{A}$ we put
$$\mathcal{K}'(A):=\sup\{\|\chi_B\|_X:\,B\in \mathcal{A}_0,\, B\subseteq A\}.$$
Similar to the proof of \cite[Theorem 14.22]{niel} with different details, one can prove that
\begin{itemize}
	\item if $C\in \mathcal{A}$ and  $\mathcal{K}(C)=0$, then for each $\epsilon>0$ there exists $A\in \mathcal{A}_0$ such that $A\subseteq C$, $0<\|\chi_A\|_X<\min\{\epsilon,\mathcal{K}'(C)\}$ and $\mathcal{K}(C-A)=0$.
    \end{itemize}
Indeed, let $\mathcal{K}(C)=0$. Then, there exists a set $B\subseteq C$ such that $0<\|\chi_B\|_X<\min\{\varepsilon,\mathcal{K}'(C)\}$. If $\mathcal{K}(C-B)=0$ we set $A:=B$. If $\mathcal{K}(C-B)>0$, by (2) we have $\mathcal{K}(B)=0$, and so there is a set $D\in\mathcal{A}_0$ such that $D\subseteq B$ and $0<\|\chi_D\|_X<\|\chi_B\|_X$. In this situation, because of (2) we have $\mathcal{K}(D)=0$ or $\mathcal{K}(B-D)=0$, and then by (1) it would be enough to set $A:=B-D$ or $A:=D$, respectively.

Now, since $\inf\{\|\chi_E\|_X: E\in \mathcal{A}_0\}=0$, we have $\mathcal{K}(\Omega)=0$. So, there exists $A_1\in \mathcal{A}_0$ such that 
$0<\|\chi_{A_1}\|_X<\min\{\frac{1}{2}, \mathcal{K}'(\Omega)\}$ and $\mathcal{K}(\Omega-A_1)=0$. Setting $C:=\Omega-A_1$ in the above fact, there exists $A_2\in \mathcal{A}_0$ such that $A_2\subsetneq \Omega-A_1$, 
$$0<\|\chi_{A_2}\|_X<\min\{\frac{1}{2^2}, \mathcal{K}'(\Omega-A_1)\},$$ 
and $\mathcal{K}(\Omega-(A_1\cup A_2))=\mathcal{K}((\Omega-A_1)-A_2)=0$. By continuing this method, the desired sequence in $(b)$ is obtained.
\end{proof}
Now, we can give one of the main results of this section.
\begin{theorem}
	Let $X$ be a solid Banach function space and $\inf\{\|\chi_E\|_X: E\in \mathcal{A}_0\}=0$.
	Also, assume that 		$\sup\{\|\chi_E\|_X:\, E \in {\mathcal A}_\infty\}<\infty$, where ${\mathcal A}_\infty:=\{E \in {\mathcal A}:\, \chi_E\in X\}$.
	Then, for each $p\geq 1$, $X^p_{{\rm r-strict}}:=X^p-\bigcup_{p<q}X^q$ is spaceable in $X^p$.
\end{theorem}
\begin{proof}
	We shall show that the conditions of Theorem \ref{32} hold with $\mathcal E:=X^p$ and $B=\bigcup_{p<q}X^q$. Note that since $\Omega\in\mathcal{A}_0$, for each $q>p$ we have $X^q\subset X^p$ thanks to \cite[Theorem 2.4]{tab5}. Clearly, $B$ is a cone because each $X^q$ is a linear space, and $X^p$ is a PCS-space by Remark \ref{pcs}.
	Also, the condition (2) in Theorem \ref{32} holds since $X$ is solid.
	For the condition (4),  let $f,g\in X^p$ with $E_f\cap E_g=\varnothing$ and $f+g\in B$. Then, there exists $q>p$ such that $f+g\in X^q$. We have
	$$|f|^q,|g|^q\leq |f|^q+|g|^q=|f+g|^q\in X,$$
	and this implies that $f,g\in X^q\subset B$.
At the end, we show that the condition (5) in Theorem \ref{32} hold. The main idea for the proof of this part comes from \cite[Theorem 3.3]{ber}.  Since $\inf\{\|\chi_E\|_X: E\in \mathcal{A}_0\}=0$, by Lemma \ref{Nielsen} 
	 there exists a sequence $\{A_n\}_{n=1}^\infty$ in $\mathcal{A}$ with pairwise disjoint terms such that $0<\|\chi_{A_n}\|_X\leq\frac{1}{2^n}$ for all $n\in \mathbb{N}$.  
	 As in \cite[Theorem 3.3]{ber}, for each $n\in\mathbb{N}$, we choose a strictly increasing sequence $\{p_{n,k}\}_{k=1}^\infty$ of natural numbers such that $k\leq p_{n,k}$ for all $n,k\in\mathbb{N}$ and the elements of family $\{\{p_{n,k}\}_{k=1}^\infty: n\in \mathbb{N}\}$ are mutually disjoint. For each $n, k, m\in \mathbb{N}$ we put
	\begin{align*}
	\alpha_{n,k}:=\frac{1}{\big  (k(\log(1+k))^2\|\chi_{A_{p_{n,k}}}\|_X\big)^{\frac{1}{p}}} \quad\text{and}\quad S_{n,m}:=\sum_{k=1}^{m}\alpha_{n,k}\chi_{A_{p_{n,k}}}.
	\end{align*}
	Since $\sum_{k=1}^{\infty}\|\alpha^p_{n,k}\chi_{A_{p_{n,k}}}\|_X=\sum_{k=1}^{\infty}\frac{1}{k\log(1+k)^2}<\infty$,    the sequence $\{|S_{n,m}|^p\}_{m=1}^\infty$ is Cauchy and so convergence in $X$ for all $n\in \mathbb{N}$.  Now, we have 
	$$\lim_{m\rightarrow \infty}|S_{n,m}|^p=\sum_{k=1}^{\infty}\alpha^p_{n,k}\chi_{A_{p_{n,k}}}$$
	in $X$ because $X$ is a PCS-space (see Remark \ref{pcs}). In particular, we have $f_n^p=\sum_{k=1}^{\infty}\alpha^p_{n,k}\chi_{A_{p_{n,k}}}\in X$, where 
	$$f_n:=\sum_{k=1}^{\infty}\alpha_{n,k}\chi_{A_{p_{n,k}}},\qquad(n\in\mathbb{N}).$$
	
	In fact, we have $\{f_n\}_{n=1}^\infty\subseteq X^p$ with $E_{f_n}\cap E_{f_m}=\varnothing$ for all distinct $m,n\in\mathbb{N}$. 
	 On the other hand, for each $q>p$,  if $f_n\in X^q$, then 
	 \begin{align*}
	 	\|f_n\|_{X^q}=&\big(\||f_n|^q\|_X\big)^{\frac{1}{q}}\\
	 	=&\left\|\sum_{k=1}^{\infty}\alpha_{n,k}^q\,\chi_{A_{p_{n,k}}}\right\|_X^{\frac{1}{q}}\\
	 	\geq& \left\|\alpha_{n,k}^q\,\chi_{A_{p_{n,k}}}\right\|_X^{\frac{1}{q}}\\
	 	\geq&\left(\frac{2^{(\frac{q}{p}-1)\,p_{n,k}}}{k^{\frac{q}{p}}\log(1+k)^{\frac{2q}{p}}}\right)^{\frac{1}{q}}\\
	 	\geq&\left(2^{(\frac{q}{p}-1)k}\right)^{\frac{1}{q}}.
	 \end{align*}
	So, since $k\in\mathbb{N}$ is arbitrary, we have $\|f_n\|_{X^q}=\infty$, a contradiction. Therefore, $\{f_n\}_{n=1}^\infty \subseteq X^p_{{\rm r-strict}}$ and the proof is complete.
\end{proof}

Next, an extension of the main part of \cite[Theorem 3 page 155]{rao} is proved.

Motivated by definition of a diffuse set for a measure (see \cite[page 46]{rao}), we initiate the following concept. For each $E\subseteq \Omega$, denote $\mathcal{A}_E:=\{A\subseteq E: A\in \mathcal{A}\}$.
\begin{definition}\label{3.3}
	A set $E\in\mathcal{A}$ is called {\it diffuse} for a Banach function space $X$ if $\chi_E\in X$ and for each $Y\in \mathcal{A}_E$ and $0\leq \alpha \leq \|\chi_{Y}\|_X$ there exists some $F\in \mathcal{A}_Y$ such that $\|\chi_{F}\|_X=\alpha$.
\end{definition}
The main idea for proof of the next result comes from \cite[Theorem 3 page 155]{rao}, but the details are different because the situation is more general.
\begin{theorem}\label{thm444}
	Let $\Phi_1, \Phi_2$ be two strictly increasing continuous Young functions. If there exists a diffuse set $E\in \mathcal{A}_\infty$ for $X$ with $\mu(E)>0$, then the inclution $X^{\Phi_2} \subseteq X^{\Phi_1}$ implies that $\Phi_1\prec \Phi_2$.
\end{theorem}
\begin{proof}
	Let the assumptions hold and $X^{\Phi_2} \subseteq X^{\Phi_1}$.	In contrast, assume that $\Phi_1\nprec \Phi_2$. Then,  there exists an increasing sequence $\{a_n\}_{n=1}^\infty$ in $(0,\infty)$ such that $\lim_{n\rightarrow \infty}a_n=\infty$ and 
	\begin{align}\label{sss}
	\Phi_1(a_n)>n\,2^n\Phi_2(n^2a_n), \qquad (n\in \mathbb{N}).
	\end{align}	
	Since $\sum_{n=1}^{\infty}\frac{\Phi_2(a_1)\|\chi_E\|_X}{2^n\Phi_2(n^2a_n)}<\|\chi_E\|_X$, there exists  $E_0\in \mathcal{A}_E$ such that
	\begin{align*}
	\|\chi_{E_0}\|_X=\sum_{n=1}^{\infty}\frac{\Phi_2(a_1)\|\chi_E\|_X}{2^n\Phi_2(n^2a_n)},
	\end{align*}
	because $E$ is a diffuse set for $X$. Inductively, one can find a pairwise disjoint sequence $\{E_n\}_{n=1}^\infty$ in $\mathcal{A}_{E_0}$ such that
	
	\begin{equation}\label{ttt}
	\|\chi_{E_n}\|_X=\frac{\Phi_2(a_1)\|\chi_E\|_X}{2^n\Phi_2(n^2a_n)}, \qquad (n\in \mathbb{N}).
	\end{equation}
	
	
	So, setting $f:=\sum_{n=1}^{\infty}na_n\chi_{E_n}$ we have
	\begin{align*}
	\sum_{n=1}^{\infty}\Phi_2(n^2a_n)\|\chi_{E_n}\|_X&=\sum_{n=1}^{\infty}\frac{\Phi_2(a_1)\|\chi_{E}\|_X}{2^n}\\
	&=\Phi_2(a_1)\|\chi_{E}\|_X< \infty.
	\end{align*}
	This implies that the sequence $$\left(\sum_{n=1}^{k}\Phi_2(n^2a_n)\chi_{E_n}\right)_k$$ is Cauchy and so convergent in $X$.  But, by Remark \ref{pcs}, the convergence point is $\Phi_2(f)=\sum_{n=1}^{\infty}\Phi_2(n^2a_n)\chi_{E_n}$. So,  $f\in X^{\Phi_2}$.
	
	On the other hand, let $\alpha>0$ be arbitrary. In contrast, let $\Phi_1(\alpha f)\in X$. Fix a number $m\in\mathbb{N}$ such that $\frac{1}{n}<\alpha$ for all $n\geq m$. Then, thanks to the relations \eqref{sss} and \eqref{ttt} we have
	\begin{align*}
	\|\Phi_1(\alpha f)\|_X=&\left\|\sum_{n=1}^{\infty}\Phi_1(\alpha na_{n})\chi_{E_{n}}\right\|_X\\
	\geq&\Phi_1(\alpha k a_{k})\|\chi_{E_{k}}\|_X\\
	\geq&\Phi_1(a_{k})\|\chi_{E_{k}}\|_X\\
	\geq&k\,\Phi_2(a_1)\|\chi_{E}\|_X.
	\end{align*}
	for all $k\geq m$, and so $\|\Phi_1(\alpha f)\|_X=\infty$, a contradiction. This shows that $f\notin X^{\Phi_1}$, and the proof is complete.	
\end{proof}
\begin{corollary}
Under the assumptions of Theorem \ref{thm444}, if $\Phi_1\nprec \Phi_2$, then $X^{\Phi_2}-X^{\Phi_1}$ is spaceable in $X^{\Phi_2}$.
\end{corollary}
\begin{proof}
	Let for each $n\in\mathbb{N}$, $N_n$ be a strictly increasing sequence of natural numbers and $\{N_n\}_{n=1}^\infty$ be a partition of $\mathbb{N}$. Then, similar to the proof of Theorem \ref{thm444} it would be routine to construct a sequence $(f_n)$ in  $X^{\Phi_2}-X^{\Phi_1}$ such that for each distinct $m,n\in\mathbb{N}$, $E_{f_n}\cap E_{f_m}=\varnothing$. Now, easily by Theorem \ref{32} the statement is proved.
\end{proof}
\section{Some New Sufficient Conditions With Applications}
In this section, first we give an abstract version of Theorem \ref{32} and then present some applications regarding Cartesian product of $X^p$ spaces and also the space of bounded linear operators on a Hilbert space. For this we need to initiate the next concept.
\begin{definition}\label{dd}
	Let $\mathcal{E}$ be a topological vector space. We say that a relation $\sim$ on $\mathcal{E}$ has property $(D)$ if the following conditions hold:
	\begin{enumerate}
		\item If $(x_n)$ is a sequence in $\mathcal{E}$ such that $x_n\sim x_m$ for all distinct index $m,n$, then for each disjoint finite subsets $A,B$ of $\mathbb{N}$ we have 
		$$\sum_{n\in A}\alpha_nx_n\sim \sum_{m\in B}\beta_m x_m,$$
		where $\alpha_n$ and $\beta_m$'s are arbitrary scalars.
		\item If a sequence $(x_n)$ converges to $x$ in $\mathcal{E}$ and for some $y\in \mathcal{E}$, $x_n\sim y$ for all $n\in\mathbb{N}$, then $x\sim y$.
	\end{enumerate}
\end{definition}
We recall that a sequence $(x_n)$ in a Banach space $\mathcal{E}$ is called a \emph{basic sequence} if for each $x$ in $\overline{{\rm span}}\{x_1,x_2,\ldots\}$, the closed linear span of $\{x_1,x_2,\ldots\}$, there are unique scalars $\alpha_1,\alpha_2,\ldots$ such that $x=\lim_n \sum_{k=1}^n\alpha_kx_k$ in $\mathcal{E}$. Note that, by \cite[Proposition 1, Chapter II]{bea}, $(x_n)$ is a basic sequence if and only if there is a constant $k>0$ such that for each $m,n$ with $m\geq n$ and each scalars $\alpha_1,\ldots,\alpha_m$, $\left\|\sum_{j=1}^n\alpha_jx_j\right\|\leq k\,\left\|\sum_{j=1}^m\alpha_jx_j\right\|$. In this paper (as in \cite{ber}) a subset $B$ of a vector space is called a \emph{cone} if for each scalar $c$, $cB\subseteq B$.
\begin{theorem}\label{322}
Let $(\mathcal{E},\|\cdot\|)$ be a Banach space, $\sim$ be a relation on $\mathcal{E}$ with property $(D)$, and $K$ be a nonempty
subset of $\mathcal{E}$. Assume that:
\begin{enumerate}	
	\item there is a constant $k>0$ such that $\|x+y\|\geq k\,\|x\|$ for all $x,y\in \mathcal{E}$ with $x\sim y$;
	\item  $K$ is a cone;
	\item if $x,y\in \mathcal{E}$ such that $x+y\in K$ and  $x\sim y$ then $x,y\in K$;
	\item there is an infinite sequence $\{x_n\}_{n=1}^\infty\subseteq \mathcal{E}-K$ such that for each distinct $m,n\in\mathbb{N}$, $x_m\sim x_n$.
\end{enumerate}
Then, $\mathcal{E}-K$ is spaceable in $\mathcal{E}$.
\end{theorem}
\begin{proof}
	The main idea of the proof comes from \cite[Theorem 2.2]{ber}. Indeed, applying condition $(D)$ in Definition \ref{dd} and thanks to \cite[Proposition 1, Chapter II]{bea} one can see that the sequence $(x_n)$ in assumption (4) is a basic sequence, and this shows that $(x_n)$ is linearly independent. Let $0\neq x\in \overline{{\rm span}}\{x_1,x_2,\ldots\}$. Then, since by definition of basic sequences, there exist unique scalars $\alpha_1,\alpha_2,\ldots$ such that $x=\sum_{n=1}^\infty \alpha_n x_n$. Put $N:=\min\{n\in\mathbb{N}:\,\alpha_n\neq 0\}$.
So, $x=\alpha_N x_N+y$, where $y:=\lim_{m\rightarrow\infty} \sum_{n=N+1}^m\alpha_n x_n$.
Again, applying both conditions in Definition \ref{dd} we have $x_N\sim y$. In contrast, if $x\in K$, then by the assumptions (3) and (2) we have $x_N\in K$, a contradiction. Therefore, $(\mathcal{E}-K)\cup \{0\}$ contains the closed infinite-dimensional space 
$\overline{{\rm span}}\{x_1,x_2,\ldots\}$, and this completes the proof.
\end{proof}
\begin{remark}
	We mention that this theorem is a generalization of \cite[Theorem 2.2]{ber} (Theorem \ref{32}). Just note that for each Banach function space $X$, the relation $\sim$ defined by 
	$$f\sim g \text{ if and only if } E_f\cap E_g=\varnothing$$
	for all $f,g\in X$, has the property $(D)$.
\end{remark}
Applying Theorem \ref{322}, we give the next result which could not be concluded from \cite[Theorem 2.2]{ber}.
\begin{theorem}
	Let $X$ be a solid Banach function space on $\Omega$ and assume that $\inf\{\|\chi_E\|_X: E\in \mathcal{A}_0\}=0$. 
	 Then, for each $1\leq p,q< r$, the set $\{(f,g)\in X^p\times X^q:\,fg\notin X^r\}$ is spaceable in $X^p\times X^q$. 
\end{theorem}
\begin{proof}
	Let $1\leq p,q< r$. By Lemma \ref{Nielsen}, there is a sequence $\{A_n\}_{n=1}^\infty$ in $\mathcal{A}_0$ with disjoint terms such that $0<\|\chi_{A_n}\|_X\leq\frac{1}{2^n}$ for all $n\in \mathbb{N}$. We define
	$$j:=\sum_{n=1}^{\infty}\|\chi_{A_n}\|_X^{\frac{-1}{r}}\cdot\chi_{A_n}.$$
	Then, similar to the proof of Lemma \ref{Nielsen}, one can see that $j\in X^p\cap X^q$. In contrast, if $j^2\in X^r$, then we have 
	\begin{align*}
	\|j^2\|_{X^r}&=\||j^2|^r\|_X\\	
	&=\left\|\left(\sum_{n=1}^{\infty}\,\|\chi_{A_n}\|_X^{\frac{-2}{r}}\cdot\chi_{A_n}\right)^r\right\|_X\\
	&\geq\left\|\left\|\chi_{A_k}\right\|_X^{-2}\cdot\chi_{A_k}\right\|_X\geq 2^k
	\end{align*}
	for all $k\in\mathbb{N}$, a contradiction. This implies that setting $$K:=\{(f,g)\in X^p\times X^q:\,fg\in X^r\},$$
	 we have $(j,j)\in (X^p\times X^q)-K$. Put $h:=j^2$. By the above relations, it would be standard to find a sequence $(F_n)$ such that for each distinct $m,n\in\mathbb{N}$, $F_n\cap F_m=\varnothing$, and $h\chi_{F_{n}}\notin X^r$. This implies that $(j\chi_{F_n},j\chi_{F_n})\in (X^p\times X^q)-K$ for all $n\in\mathbb{N}$. Finally, note that the relation $\sim$ defined by 
 $$(f_1,g_1)\sim (f_2,g_2) \text{ if and only if } E_{f_1}\cap E_{f_2}=E_{g_1}\cap E_{g_2}=\varnothing$$
 for all $f_i\in X^p$ and $g_i\in X^q$ ($i=1,2$), satisfies the condition $(D)$. Applying Theorem \ref{322}, the proof is complete.
\end{proof}
Let $\mathcal{H}$ be a separable infinite dimensional Hilbert space and $\{e_j\}_{j\in\mathbb{N}}$ be an orthonormal basis for $\mathcal{H}$. For each non-empty subset $J\subseteq \mathbb{N}$ we let $P_J$ denote the orthogonal projection onto $\overline{{\rm span}}\{e_j\}_{j\in J}$. For each $T,S\in B(\mathcal{H})$, the space of all bounded linear operators on $\mathcal{H}$, we say that $T\sim S$ if there exist two disjoint subsets $J_1,J_2\subseteq \mathbb{N}$ such that $P_{J_1}TP_{J_2}=T$ and $P_{J_2}SP_{J_2}=S$. With these notations, we give the next lemma.
\begin{lemma}\label{bhlemma}
	The relation $\sim$ on $B(\mathcal{H})$ defined above has the property (D).
\end{lemma}
\begin{proof}
	Suppose that $\{T_n\}_{n\in\mathbb{N}}$ is a sequence in $B(\mathcal{H})$ such that for each distinct $m,n$ we have $T_n\sim T_m$. Let $A:=\{n_1,\ldots,n_k\}$ and $B:=\{m_1,\ldots,m_l\}$ be two disjoint finite subsets of $\mathbb{N}$. Then, for each $n\in A$ and $m\in B$, there exist some disjoint subsets $J_{(n,m)},J'_{(n,m)}\subseteq \mathbb{N}$ such that 
	\begin{equation}\label{qqq}
	P_{J_{(n,m)}}T_nP_{J_{(n,m)}}=T_n\qquad\text{and}\qquad P_{J_{(n,m)}'}T_mP_{J_{(n,m)}'}=T_m.
	\end{equation}
	  By \cite[Chapter 2, Section 8, Theorem 4]{bir}, we have 
	\begin{equation}\label{lll}
	P_{\bigcap_{r=1}^lJ_{(n,m_r)}}=P_{J_{(n,m_1)}}P_{J_{(n,m_2)}}\ldots P_{J_{(n,m_l)}}
	\end{equation}
	 for all $n\in A$. Then, \eqref{qqq} implies that 
	\begin{equation}\label{mmm}
	P_{\bigcap_{r=1}^lJ_{(n,m_r)}}T_nP_{\bigcap_{r=1}^lJ_{(n,m_r)}}=T_n
	\end{equation}
	for all $n\in A$. Indeed, by \eqref{lll} we have 
		\begin{align*}
		P_{\bigcap_{r=1}^lJ_{(n,m_r)}}T_n&P_{\bigcap_{r=1}^lJ_{(n,m_r)}}\\
		&=P_{J_{(n,m_l)}}P_{J_{(n,m_{l-1})}}\ldots P_{J_{(n,m_1)}}T_nP_{J_{(n,m_1)}}P_{J_{(n,m_2)}}\ldots P_{J_{(n,m_l)}}\\
		&=T_n
		\end{align*}
	for all $n\in A$. Put
	$$E:=\bigcup_{i=1}^k\bigcap_{r=1}^lJ_{(n_i,m_r)}.$$
	Then, by \cite[Chapter 2, Section 8, Corollary 5]{bir}, we have 
	$$P_EP_{\bigcap_{r=1}^lJ_{(n_i,m_r)}}=P_{\bigcap_{r=1}^lJ_{(n_i,m_r)}}P_E=P_{\bigcap_{r=1}^lJ_{(n_i,m_r)}}$$
	for every $i\in\{1,\ldots,k\}$. Because of \eqref{mmm}, 
	$$P_ET_{n_i}P_E=T_{n_i}$$ for all $i\in\{1,\ldots,k\}$. This implies that 
	$P_E(\sum_{i=1}^k\alpha_i\,T_{n_i})P_E=\sum_{i=1}^k\alpha_i\,T_{n_i}$ for all $\alpha_1,\ldots,\alpha_k\in\mathbb{C}$. Similarly, setting $F:=\bigcup_{r=1}^l\bigcap_{i=1}^kJ'_{(n_i,m_r)}$ we have 
	$P_F(\sum_{r=1}^l\beta_i\,T_{m_r})P_F=\sum_{r=1}^l\beta_i\,T_{m_r}$
	for all $\beta_1,\ldots,\beta_l\in\mathbb{C}$. Now, since $J_{(n_i,m_r)}\cap J_{(n_i,m_r)}'=\varnothing$ for each $i\in\{1,\ldots,k\}$ and $r\in\{1,\ldots,l\}$, easily we have $E\cap F=\varnothing$. Therefore, $\sim$ satisfies the condition (1) in Definition \ref{dd}. Next, suppose that $S\in B(\mathcal{H})$ and  $\{T_n\}_{n\in\mathbb{N}}$ is a sequence in $B(\mathcal{H})$ such  that $T_n\rightarrow T$, in operator norm, for some $T\in B(\mathcal{H})$,  and $T_n\sim S$ for all $n$. Then, for each $n\in\mathbb{N}$, there are disjoint subsets $J_n,J_n'\subseteq \mathbb{N}$ such that
	\begin{equation}\label{zzz}
	P_{J_n}T_nP_{J_n}=T_n\qquad\text{and}\qquad P_{J_n'}SP_{J_n'}=S.
	\end{equation}
	  Again, by \cite[Chapter 2, Section 8, Theorem 4]{bir} for each $n\in\mathbb{N}$ we have 
	\begin{equation}\label{nnn}
	P_{\bigcap_{m=1}^nJ_m'}SP_{\bigcap_{m=1}^nJ_m'}=P_{J_n'}\ldots P_{J_1'}SP_{J_1'}\ldots P_{J_n'}=S.
	\end{equation}
		 Now, the sequence $\{P_{\cap_{m=1}^nJ_m'}\}_{n\in\mathbb{N}}$ is a non-increasing sequence of orthogonal projections, hence by \cite[Chapter 2, Section 8, Theorem 6]{bir},
				$$s-\lim_{n\rightarrow\infty}P_{\cap_{m=1}^nJ_m'}=P_{\cap_{m=1}^\infty J_m'},$$
				where $s-\lim$ means the limit in the strong operator topology.
				From \eqref{nnn} and thanks to \cite[Chapter 2, Section 5, Theorem 2]{bir} we have 
				\begin{align*}
				S&=s-\lim_{n\rightarrow\infty}(P_{\cap_{m=1}^nJ_m'}SP_{\cap_{m=1}^nJ_m'})\\
				&=(s-\lim_{n\rightarrow\infty}P_{\cap_{m=1}^nJ_m'})S(s-\lim_{n\rightarrow\infty}P_{\cap_{m=1}^nJ_m'})\\
				&=P_{\cap_{m=1}^\infty J_m'}SP_{\cap_{m=1}^\infty J_m'}.
				\end{align*}
By \eqref{zzz} and 	\cite[Chapter 2, Section 8, Corollary 5]{bir} we have  
$$P_{\bigcup_{m=1}^\infty J_m}T_n P_{\bigcup_{m=1}^\infty J_m}=T_n$$
for all $n\in\mathbb{N}$. Letting $n\rightarrow\infty$ we get 
$$P_{\bigcup_{m=1}^\infty J_m}T P_{\bigcup_{m=1}^\infty J_m}=T,$$
and this completes the proof because $(\bigcup_{m=1}^\infty J_m)\cap (\bigcap_{m=1}^\infty J_m')=\varnothing$.	
\end{proof}
\begin{definition}
	Let $K$ be a cone in $B(\mathcal{H})$. We denote 
	\begin{equation}
	\tilde{K}:=\bigcup_{J\subseteq \mathbb{N}}P_JKP_J,
	\end{equation}
	where $P_JKP_J:=\{P_JTP_J:\,T\in K\}$.
\end{definition}
Note that if $K$ is a cone, then $\tilde{K}$ is a cone as well. Moreover, $P_J\tilde{K}P_J\subseteq\tilde{K}$ for all $J\subseteq \mathbb{N}$, and in particular, $K\subseteq \tilde{K}$.
\begin{theorem}\label{lemma}
	Let $K$ be a cone in $B(\mathcal{H})$. If there exists a sequence of mutually disjoint subsets $\{J_n\}_{n\in\mathbb{N}}$ of $\mathbb{N}$ satisfying that $P_{J_n}\tilde{K}P_{J_n}\neq P_{J_n}B(\mathcal{H})P_{J_n}$ for all $n\in\mathbb{N}$, then $B(\mathcal{H})-\tilde{K}$ (and consequently $B(\mathcal{H})-K$) is spaceable in $B(\mathcal{H})$. The statement holds if we consider $B_0(\mathcal{H})$ instead of $B(\mathcal{H})$.
\end{theorem}
\begin{proof}
	We show that the relation $\sim$ defined before Lemma \ref{bhlemma} satisfies the conditions in Theorem \ref{322} regarding the cone $\tilde{K}$.
Suppose that $T,S\in B(\mathcal{H})$ with $T\sim S$. Then, there exist disjoint subsets $J,J'\subseteq \mathbb{N}$ such that $P_JTP_J=T$ and $P_{J'}SP_{J'}=S$. By disjointness of $J$ and $J'$, from \cite[Chapter 2, Section 8, Theorem 2]{bir} we have 
$P_JSP_J=P_{J}P_{J'}SP_{J'}P_{J}=0$. We get 
\begin{align*}
\|T+S\|&\geq \|P_J\|\,\|T+S\|\,\|P_J\|\\
&\geq \|P_J(T+S)P_J\|\\
&=\|P_JTP_J\|\\
&=\|T\|.
\end{align*}
This shows that the relation $\sim$ satisfies the condition (1) of Theorem \ref{322}. Now, if in addition $T+S\in \tilde{K}$, we have
\begin{align*}
T&= P_J TP_J\\
&=P_J(T+S)P_J\in \tilde{K}.
\end{align*}
Similarly, $S\in \tilde{K}$. So, the condition (2) in Theorem \ref{322} holds with respect to the cone $\tilde{K}$. Finally, consider the sequence $\{J_n\}$ of mutually disjoint subsets of $\mathbb{N}$ which was described in the assumptions. So, for each $n$ we can choose an operator $T_n\in P_{J_n}B(\mathcal{H})P_{J_n}-P_{J_n}\tilde{K}P_{J_n}$. Then, easily  one can see that $\{T_n\}_{n\in\mathbb{N}}\subset B(\mathcal{H})-\tilde{K}$ and  for each distinct $m,n\in\mathbb{N}$ we have 
$$T_n=P_{J_n}T_nP_{J_n}\sim P_{J_m}T_mP_{J_m}=T_m,$$
and this completes the proof.
\end{proof}
\begin{corollary}
	The set of all bounded linear operators on $\mathcal{H}$ which are not positive-semidefinite, is spaceable in $B(\mathcal{H})$.
\end{corollary}
\begin{proof}
	Let $K$ be the set of all scalar multiplicative of positive semidefinite operators on $\mathcal{H}$. Then, $K$ is a cone and $P_JKP_J\subseteq K$	for all $J\subseteq \mathbb{N}$ and so $\tilde{K}=K$. By some calculations one can see that the assumptions of Theorem \ref{lemma} hold in this situation, and therefore $B(\mathcal{H})-K$ is spaceable. This implies that $B(\mathcal{H})-B_0(\mathcal{H})$ is spaceable as well.
\end{proof}
The following result is directly concluded from Theorem \ref{lemma}.
\begin{corollary}\label{cor49}
	If $K$ is a two sided ideal cone in $B(\mathcal{H})$ and there exists a sequence of mutually disjoint subsets $\{J_n\}_{n\in\mathbb{N}}$ of $\mathbb{N}$ satisfying the condition $P_{J_n}KP_{J_n}\neq P_{J_n}B(\mathcal{H})P_{J_n}$ for all $n\in\mathbb{N}$, then $B(\mathcal{H})-K$ is spaceable in $B(\mathcal{H})$.
\end{corollary}
\begin{example}\label{example}
		$B_0(\mathcal{H})$, the space of all compact operators on $\mathcal{H}$, is a two sided ideal cone in $B(\mathcal{H})$ which satisfies the requirements of  Corollary \ref{cor49}. So, the set of all non-compact operators is a spaceable subset of $B(\mathcal{H})$.
\end{example}
\begin{remark}
	By the same argument, $B_0(\mathcal{H})-(B_+(\mathcal{H})\bigcap B_0(\mathcal{H}))$ is spaceable in $B_0(\mathcal{H})$, where $B_+(\mathcal{H})$ is the set of all positive semidefinite operators on $\mathcal{H}$.  One can also replace $B_0(\mathcal{H})$ by the real Banach space of all Hermitian operators on $\mathcal{H}$.
\end{remark}
\begin{remark}
	Let $(B_1(\mathcal{H}),\|\cdot\|_1)$ and $(B_2(\mathcal{H}),\|\cdot\|_2)$ denote the Banach space of all trace-class operators equipped with the trace norm and the Banach space of all Hilbert-Schmidt operators equipped with the Hilbert-Schmidt norm, respectively. Since for every $T\in B(\mathcal{H})$, $S\in B_1(\mathcal{H})$ and $G\in B_2(\mathcal{H})$ we have $\|ST\|_1,\|TS\|_1\leq \|T\|\,\|S\|_1$ and 
	$\|TG\|_2,\|GT\|_2\leq \|T\|\,\|G\|_2$, it is not hard to see that the similar argument as in the proof of Lemma  \ref{bhlemma} and Theorem \ref{lemma} can be applied to deduce that $B_1(\mathcal{H})-(B_+(\mathcal{H})\bigcap B_1(\mathcal{H}))$ and $B_2(\mathcal{H})-(B_+(\mathcal{H})\bigcap B_2(\mathcal{H}))$ are spaceable in $B_1(\mathcal{H})$ and $B_2(\mathcal{H})$, respectively. Also, 	$B_0(\mathcal{H})-B_1(\mathcal{H})$ and $B_0(\mathcal{H})-B_2(\mathcal{H})$  are both spaceable in $B_0(\mathcal{H})$.
\end{remark}

\bibliographystyle{amsplain}

\end{document}